\documentclass[a4paper,12pt]{amsart}
\usepackage[english]{babel}
\usepackage{amsmath,amssymb,amsthm}

\usepackage{graphicx}

\usepackage[pdftex]{color}
\usepackage[bookmarks=true,hyperindex,pdftex,colorlinks, citecolor=blue]{hyperref}
\hypersetup{pdftitle={Bishop-Phelps-Bollob\'{a}s version of Lindenstrauss properties A and B},
pdfauthor={Richard Aron, Yun Sung Choi, Sun Kwang Kim, Han Ju Lee, and Miguel Martin},pdfpagemode=UseOutlines}

\theoremstyle{plain}
\newtheorem{theorem}{Theorem}[section]
\newtheorem{corollary}[theorem]{Corollary}
\newtheorem{prop}[theorem]{Proposition}
\newtheorem{proposition}[theorem]{Proposition}
\newtheorem{lemma}[theorem]{Lemma}

\theoremstyle{definition}

\newtheorem{example}[theorem]{Example}

\newtheorem{definition}[theorem]{Definition}

\newtheorem{notation}[theorem]{Notation}

 \DeclareMathOperator{\re}{Re\,}

 \DeclareMathOperator{\dist}{dist\,}

\newcommand{\K}{\mathbb{K}}

\newcommand{\C}{\mathbb{C}}
\newcommand{\R}{\mathbb{R}}
\newcommand{\N}{\mathbb{N}}

\newcommand{\eps}{\varepsilon}

\newcommand{\lra}{\longrightarrow}

\renewcommand{\leq}{\leqslant}
\renewcommand{\geq}{\geqslant}

\begin{document}

\title[Bishop-Phelps-Bollob\'{a}s version of Lindenstrauss properties A and B]{The Bishop-Phelps-Bollob\'{a}s version of \\  Lindenstrauss properties A and B}

\author[Aron]{Richard Aron}
\address[Aron]{Department of Mathematical Sciences, Kent State University, Kent, OH 44242, USA}
\email{\texttt{aron@math.kent.edu}}

\author[Choi]{Yun Sung Choi}
\address[Choi]{Department of Mathematics, POSTECH, Pohang (790-784), Republic of Korea}
\email{\texttt{mathchoi@postech.ac.kr}}

\author[Kim]{Sun Kwang Kim}
\address{Department of Mathematics, Kyonggi University, Suwon 443-760, Republic of Korea}
\email{sunkwang@kgu.ac.kr}

\author[Lee]{Han Ju Lee}
\address[Lee]{Department of Mathematics Education,
Dongguk University - Seoul, 100-715 Seoul, Republic of Korea}
\email{\texttt{hanjulee@dongguk.edu}}

\author[Mart\'{\i}n]{Miguel Mart\'{\i}n}
\address[Mart\'{\i}n]{Departamento de An\'{a}lisis Matem\'{a}tico,
Facultad de Ciencias,
Universidad de Granada,
E-18071 Granada, Spain}
\email{\texttt{mmartins@ugr.es}}

\dedicatory{Dedicated to the memory of Joram Lindenstrauss and Robert Phelps}

\date{May 27th, 2013; Revised August 27th, 2014}

\subjclass[2000]{Primary 46B20; Secondary 46B04, 46B22}

\keywords{Banach space, approximation, norm-attaining operators, Bishop-Phelps-Bollob\'{a}s theorem.}

\thanks{First author partially supported by Spanish MICINN and FEDER Project MTM2008-03211.
Second  author supported by Basic Science Research Program
through the National Research Foundation of Korea (NRF) funded by the
Ministry of Education (No.2010-0008543 and No. 2013053914).
Fourth author partially supported by Basic Science Research Program through the National Research Foundation of Korea(NRF) funded by the Ministry of Education, Science and Technology (NRF-2012R1A1A1006869).
Fifth author partially supported by Spanish MICINN and FEDER project no.~MTM2012-31755, Junta de Andaluc\'{\i}a and FEDER grants FQM-185 and P09-FQM-4911, and by ``Programa Nacional de Movilidad de Recursos Humanos del Plan Nacional de I+D+i 2008--2011'' of the Spanish MECD}

\begin{abstract}
We study a Bishop-Phelps-Bollob\'{a}s version of Lindenstrauss properties A and B. For domain spaces, we study Banach spaces $X$ such that $(X,Y)$ has the Bishop-Phelps-Bollob\'{a}s property (BPBp) for every Banach space $Y$. We show that in this case, there exists a universal function $\eta_X(\eps)$ such that for every $Y$, the pair $(X,Y)$ has the BPBp with this function. This allows us to prove some necessary isometric conditions for $X$ to have the property. We also prove that if $X$ has this property in every equivalent norm, then $X$ is one-dimensional. For range spaces, we study Banach spaces $Y$ such that $(X,Y)$ has the Bishop-Phelps-Bollob\'{a}s property for every Banach space $X$. In this case, we show that there is a universal function $\eta_Y(\eps)$ such that for every $X$, the pair $(X,Y)$ has the BPBp with this function. This implies that this property of $Y$ is strictly stronger than Lindenstrauss property B. The main tool to get these results is the study of the Bishop-Phelps-Bollob\'{a}s property for $c_0$-, $\ell_1$- and $\ell_\infty$-sums of Banach spaces.
\end{abstract}

\maketitle

\section{Introduction}

In his seminal work of 1963, J.~Lindenstrauss \cite{Lindens} examined the extension of the Bishop-Phelps theorem, on denseness
of the family of norm-attaining scalar-valued functionals on a Banach space, to vector-valued linear operators. In this paper,
he introduced two universal properties, $A$ and $B,$ that a Banach space might have. (All notions mentioned here will be
reviewed later in this Introduction.) Seven years later, B.~Bollob\'as observed
that there is a numerical version of the Bishop-Phelps theorem, and his contribution is known as the Bishop-Phelps-Bollob\'as
theorem. Recently, vector-valued versions of this result have been studied (see, e.g., \cite{AAGM2}). Our goal here is to
introduce and study analogues of properties $A$ and $B$ in the context of vector-valued versions of the Bishop-Phelps-Bollob\'as
theorem.\\

Let us now restart the Introduction, this time giving the necessary background material to help make the paper entirely accessible.
The Bishop-Phelps-Bollob\'{a}s property was introduced in 2008
\cite{AAGM2} as an extension of the Bishop-Phelps-Bollob\'{a}s
theorem to the vector valued case. It can be regarded as a
``quantitative version'' of the study of norm-attaining operators
initiated by J.~Lindenstrauss in 1963. We begin with some notation
to present its definition. Let $X$ and $Y$ be Banach spaces over the
field $\K=\R~\mbox{or}~\C$. We will use the common notation $S_X$,
$B_X$, $X^*$ for the unit sphere, the closed unit ball and the dual
space of $X$ respectively, $L(X,Y)$ for the Banach space of all
bounded linear operators from $X$ into $Y,$ and $NA(X,Y)$ for the
subset of all norm-attaining operators. (We say that an operator
$T\in L(X,Y)$ \emph{attains its norm} if $\|T\|=\|Tx\|$ for some
$x\in S_X$).  We will abbreviate $L(X,X)$,
resp.\ $NA(X,X)$, by $L(X)$, resp.\ $NA(X)$.

\begin{definition}[\textrm{\cite[Definition~1.1]{AAGM2}}]
A pair of Banach spaces $(X,Y)$ is said to have the
\emph{Bishop-Phelps-Bollob\'{a}s property} (\emph{BPBp} for short)
if for every $\eps\in (0,1)$ there is $\eta(\eps)>0$ such that for every
$T_0\in L(X,Y)$ with $\|T_0\|=1$ and every $x_0\in S_X$ satisfying
$$
\|T_0(x_0)\|>1-\eta(\eps),
$$
there exist $S\in L(X,Y)$ and $x\in S_X$ such that
$$
1=\|S\|=\|Sx\|,\qquad \|x_0-x\|<\eps \qquad \text{and} \qquad \|T_0-T\|<\eps.
$$
In this case, we will say that $(X,Y)$ has the BPBp with function $\eps\longmapsto \eta(\eps)$.
\end{definition}

The study of the ``denseness of norm-attaining things'' goes back to the celebrated Bishop-Phelps
theorem \cite{BishopPhelps} which appeared in 1961. This theorem simply states
that $NA(X,\K)$ is dense in $X^*$ for every Banach space $X$. The problem of the
denseness of $NA(X,Y)$ in $L(X,Y)$ was given in that paper, and in 1963
J.~Lindenstrauss \cite{Lindens} provided a simple example to show that it is not true in general. On the other hand,
motivated by some problems in numerical range theory, in 1970 B.~Bollob\'{a}s \cite{Bollobas} gave
a ``quantitative version'' of the Bishop-Phelps theorem  which, in the above language, states that the
pair $(X,\K)$ has the BPBp for every Banach space $X$. Nowadays, this result is known as the
Bishop-Phelps-Bollob\'{a}s theorem. We refer to the expository paper \cite{Acosta-RACSAM} for a
detailed account on norm-attaining operators; any result not explicitly referred to below can be found there.

As the problem of the denseness of norm attaining operators is so general, J. Lindenstrauss \cite{Lindens}
introduced and studied the following two properties. A Banach space $X$ is said to have \emph{Lindenstrauss property A}
if $\overline{NA(X,Z)}=L(X,Z)$ for every Banach space $Z$. A Banach space $Y$ is said to have
\emph{Lindenstrauss property B} if $\overline{NA(Z,Y)}=L(Z,Y)$ for every Banach space $Z$. We remark
that both properties A and B are isometric in nature (that is, they depend upon the particular norm).

Lindenstrauss property A is trivially satisfied by finite-dimensional spaces by compactness of the closed unit ball, and it is also satisfied by reflexive spaces \cite{Lindens}. J.~Bourgain \cite{Bou} proved in 1977 that the
Radon-Nikod\'{y}m property characterizes Lindenstrauss property A isomorphically: a Banach space has the
Radon-Nikod\'{y}m property if and only if it has Lindenstrauss property A in every equivalent norm. From the isometric viewpoint, W.~Schachermayer \cite{Schacher} introduced property $\alpha,$ which
implies Lindenstrauss property A. It is satisfied, for instance, by $\ell_1$, and it is also satisfied in many
Banach spaces, including all separable spaces, after an equivalent renorming.  There is a weakening of
property $\alpha$, called \emph{property quasi-$\alpha$} introduced by Y.~S.~Choi and H.~G.~Song \cite{CS-quasialpha}
which  still implies Lindenstrauss property A. On the other hand, examples of Banach spaces
failing Lindenstrauss property A are $c_0$, non-atomic $L_1(\mu)$ spaces, $C(K)$ spaces for infinite
and metrizable $K$, and the canonical predual of any Lorentz sequence space $d(w,1)$ with
$w\in \ell_2\setminus \ell_1$. Finally, we mention that Lindenstrauss property A is stable under arbitrary
$\ell_1$-sums \cite{CS-quasialpha}.

Much less is known about Lindenstrauss property B. The base field $\K$ clearly has it, since this is just the Bishop-Phelps theorem. However, it is unknown whether every finite-dimensional space has Lindenstrauss property B,
even for two-dimensional Euclidean space. J.~Lindenstrauss gave an isometric sufficient condition
for property B, called \emph{property $\beta$}, which is satisfied by polyhedral finite-dimensional spaces, and by
any space between $c_0$ and $\ell_\infty$ inclusive. It was proved later by J.~Partington \cite{Partington} that
every Banach space can be equivalently renormed to have property $\beta$ and, therefore, to have Lindenstrauss
property B. There is a weakening of property $\beta$, called \emph{property quasi-$\beta$} introduced by
M.~Acosta, F.~Aguirre and R.~Pay\'{a} \cite{AAP-quasibeta} which provides new examples of spaces with Lindenstrauss
property B, such as some non-polyhedral finite-dimensional spaces and the canonical predual of any Lorentz sequence
space $d(w,1)$ with $w\in c_0\setminus \ell_1$. Among spaces without Lindenstrauss property B we find infinite-dimensional
$L_1(\mu)$ spaces, $C[0,1]$, $d(w,1)$ with $w\in \ell_2\setminus \ell_1$ and every infinite-dimensional strictly convex
space (or $\C$-rotund space in the complex case); in particular $\ell_p$ for $1<p<\infty$ all fail Lindenstrauss
property B. Let us finish by mentioning that Lindenstrauss property B is stable under arbitrary $c_0$-sums \cite[Proposition~3]{AAP-quasibeta}.

Since its introduction in 2008, quite a few papers regarding the Bishop-Phelps-Bollob\'{a}s property have been
published (see, e.g., \cite{ABGM,ACGM-Adv,CasGuiKad,CS,KimLee}). Among others, the following pairs
have been shown to have the BPBp: $(L_1(\mu),L_\infty[0,1])$ for every $\sigma$-finite measure $\mu$, $(X,C_0(L))$ for every Asplund
space $X$ and every Hausdorff locally compact space $L$, $(X,Y)$ when $X$ is uniformly convex or $Y$ has property
$\beta$, or both $X$ and $Y$ are finite-dimensional.

In this paper, we will deal with the following definitions, which are exactly the BPB versions of Lindenstrauss properties A and B.

\begin{definition} Let $X$ and $Y$ be Banach spaces.
We say that $X$ is a \emph{universal BPB domain space} if for every Banach space $Z$, the pair $(X,Z)$ has the BPBp.
We say that $Y$ is a \emph{universal BPB range space} if for every Banach space $Z$, the pair $(Z,Y)$ has the BPBp.
\end{definition}

The following assertions are clearly true:
\begin{enumerate}
\item[(a)] a universal BPB domain space has Lindenstrauss property A,
\item[(b)] a universal BPB range space has Lindenstrauss property B.
\end{enumerate}

The converse of (a) is known to be false: the space $\ell_1$ has
Lindenstrauss property A but fails to be a universal BPB domain
space \cite[Theorem~ 4.1]{{AAGM2}}.  Even more, every finite-dimensional Banach space clearly has
Lindenstrauss property A, but $\ell_1^2$ fails to be a
universal BPB domain space. (We will prove this later in
Corollary~\ref{cor:ell12-strictlyconvex}, but it can be found by ``surfing''
into the details of the proofs in \cite{AAGM2}.) Also, a two-dimensional real space is a universal
BPB domain space if and only if it is uniformly convex \cite[Corollary~9]{KimLee} (see Corollary~\ref{tool123} below).

The validity of the converse of (b) has been pending from the beginning of the study of the BPBp, since
the basic examples of spaces with Lindenstrauss property B, i.e.\ those having property $\beta$, are actually universal
BPB range spaces \cite[Theorem~2.2]{AAGM2}. In \S 4, we will provide an example of a Banach space having Lindenstrauss property B but
failing to be a universal BPB range space.

\medskip

One of the main tools in this paper is to compare the function
$\eta(\eps)$ appearing in the definition of the BPBp for
different pairs of spaces. The following notation will make clear
what we mean.

\begin{notation}
Fix a pair $(X,Y)$ of Banach spaces and write
$$
\Pi(X,Y)= \{(x,T)\in X\times L(X,Y)\,:\, \|T\|=\|x\|=\|Tx\|=1\}.
$$
For every $\rho\in (0,1),$ let
$$
{\mathcal S}(X,Y)(\rho):=\bigl\{(S,x)\in L(X,Y)\times X \, :\, \|S\|=\|x\|=1, \ \|Sx\|> 1-\rho\bigr\} .
$$
Also, for every $\eps\in (0,1)$ we define $\eta(X,Y)(\eps)$ to be the supremum of the set consisting of $0$ and
those $\rho>0$ such that for all pairs $(T_0,x_0)\in \mathcal{S}(X,Y)(\rho)$, there exists a pair $(T,x)\in \Pi(X,Y)$ such that $\|x_0-x\|<\eps$ and $\|T_0 - T\|<\eps$. Equivalently,
\begin{equation*}
\eta(X,Y)(\eps)=\inf\bigl\{1-\|Tx\|\,:\ x\in S_X,\, T\in L(X),\, \|T\|=1,\  \dist\bigl((x,T),\Pi(X,Y)\bigr)\geq\eps\bigr\},
\end{equation*}
where $\dist\bigl((x,T),\Pi(X,Y)\bigr)= \inf\bigl\{\max\{\|x-y\|,\|T-S\|\}\ :\ (y,S)\in \Pi(X,Y)\bigr\}$.
\end{notation}

It is clear that the pair $(X,Y)$ has the BPBp if and only if
$\eta(X,Y)(\eps)>0$ for every $\eps\in (0,1)$. By construction, if a function $\eps\longmapsto
\eta(\eps)$ is valid in the definition of the BPBp for the pair $(X,Y)$, then $\eta(\eps)\leq
\eta(X,Y)(\eps)$. That is, $\eta(X,Y)(\eps)$ is the best function (i.e.\ the largest) we can find to ensure
that $(X,Y)$ has the BPBp. It is also immediate that $\eta(X,Y)(\eps)$ is increasing with respect to $\eps$.

\vspace{.2in}

We now present the main results of this paper. We first study the
behavior of the BPBp with respect to direct sums of Banach spaces.
Specifically, we prove that given two families $\{X_i\, :\, i \in
I\}$ and $\{Y_j\, :\, j\in J\}$ of Banach spaces, if $X$ is the $c_0$-, $\ell_1$- or $\ell_\infty$-sum of
$\{X_i\}$ and $Y$ is the $c_0$-, $\ell_1$- or $\ell_\infty$-sum of $\{Y_j\}$, then
\begin{equation*}
\eta(X,Y)(\eps)\leq \eta(X_i,Y_j)(\eps) \qquad \bigl(i\in I,\ j\in J\bigr).
\end{equation*}
Therefore, if the pair $(X,Y)$ has the BPBp, then every pair
$(X_i,Y_j)$ does with a non-worse function $\eta$. The main
consequence of this result is that every universal BPB space has a
``universal'' function $\eta$. That is, if $X$ is a universal BPBp
domain space, then
$$
\inf\{\eta(X,Z)(\eps)\,:\,Z\ \text{Banach space}\}>0 \qquad \bigl(\eps\in(0,1)\bigr);
$$
if $Y$ is a universal BPBp range space, then
$$
\inf\{\eta(Z,Y)(\eps)\,:\,Z\ \text{Banach space}\}>0 \qquad \bigl(\eps\in(0,1)\bigr).
$$
With this fact in mind, we see from \cite{KimLee} (where the
existence of a universal function was hypothesized) some necessary
conditions for a Banach space to be a universal BPB domain space. In particular, a real, two-dimensional
universal BPB domain space must be uniformly convex. Another related result is that for Banach spaces $X$ and
$Y$ and a compact Hausdorff space $K$, if the pair $(X,C(K,Y))$ has the BPBp, then so does $(X,Y)$. We also
provide a new result on stability of density of norm attaining operators from which we can deduce that $NA(X)$
is not dense in $L(X)$ for the space $X=C[0,1]\oplus_1 L_1[0,1]$.

With respect to domain spaces, we obtain the following result: If a Banach space
contains a non-trivial $L$-summand, then it is not a universal BPB domain space. Hence every Banach space of dimension
greater than one can be equivalently renormed so as to not be a universal BPB domain space.

With respect to range spaces, we provide an example of a Banach space $\mathcal{Y}$ which is
the $c_0$-sum of a family of polyhedral spaces of dimension two such that $(\ell_1^2,\mathcal{Y})$ does not
have the BPBp. This is the announced result of a Banach
space having Lindenstrauss property B (in fact, even having property quasi-$\beta$ of \cite{AAP-quasibeta}) and
yet failing to be a universal BPB range space. It also follows that being a universal BPB range space is not stable under
infinite $c_0$- or $\ell_\infty$-sums.

\vspace{.2in}

The outline of the paper is as follows. Section~\ref{sec:tools}
contains the results on direct sums which will be the main tools for the rest of the paper. We give in
Section~\ref{sec:domain} the results related to
domain spaces, while our results related to range spaces are
discussed in Section~\ref{sec:range}.

\section{The tools: Bishop-Phelps-Bollob\'{a}s and direct sums}\label{sec:tools}

Our aim in this section is to study the relationship between the
BPBp and certain direct sums of Banach spaces. We will use these
results later to get the main theorems of the paper.

Our principal interest in this section deals with $c_0$-, $\ell_1$-,
and $\ell_\infty$-sums of Banach spaces. We need some notation.
Given a family $\{X_\lambda\, : \, \lambda\in\Lambda\}$ of Banach
spaces, we denote by $\left[\bigoplus_{\lambda\in\Lambda}
X_\lambda\right]_{c_0}$ (resp.\ $\left[\bigoplus_{\lambda\in\Lambda}
X_\lambda\right]_{\ell_1}$, $\left[\bigoplus_{\lambda\in\Lambda}
X_\lambda\right]_{\ell_\infty}$) the $c_0$-sum (resp.\ $\ell_1$-sum,
$\ell_\infty$-sum) of the family. In case $\Lambda$ has just two
elements, we use the simpler notation $X\oplus_\infty Y$ and
$X\oplus_1 Y$.

\begin{theorem}\label{th:main-sums-together}
Let $\{X_i\, :\, i \in I\}$ and $\{Y_j\, :\, j\in J\}$ be families
of Banach spaces, let $X$ be the $c_0$-, $\ell_1$-, or
$\ell_\infty$-sum of $\{X_i\}$ and let $Y$ be the $c_0$-, $\ell_1$-,
or $\ell_\infty$-sum of $\{Y_j\}$. If the pair $(X,Y)$ has the BPBp
with $\eta(\eps)$, then the pair $(X_i,Y_j)$ also has the BPBp with
$\eta(\eps)$ for every $i\in I$, $j\in J$. In other words,
$$
\eta(X,Y)\leq \eta(X_i,Y_j)\qquad \bigl(i\in I,\ j\in J\bigr).
$$
\end{theorem}

This theorem will be obtained by simply combining Propositions
\ref{fam}, \ref{fam2}, and \ref{fam3} below. In some cases, we
are able to provide partial converses.

Before providing these propositions, and therefore the proof of
Theorem~\ref{th:main-sums-together}, we present its main
consequence: universal BPB spaces have ``universal'' functions
$\eta$. In subsequent sections, we will frequently appeal to the
following result.

\begin{corollary}\label{corollary:universal-eta} Let $X$ and $Y$ be Banach spaces.
\begin{enumerate}
\item[(a)] If $X$ is a universal BPB
domain space, then there is a function $\eta_X:(0,1)\longrightarrow
\R^+$ such that for every Banach space $Y$, $(X,Y)$ has the BPBp
with $\eta_X$. In other words, for every $Y$, $\eta(X,Y)\geq
\eta_X$.
\item[(b)] If $Y$ is a universal BPB range space, then there is
a function $\eta_Y:(0,1)\longrightarrow \R^+$ such that for every
Banach space $X$, $(X,Y)$ has the BPBp with $\eta_Y$. In other
words, for every $X$, $\eta(X,Y)\geq \eta_Y$.
\end{enumerate}
\end{corollary}

\begin{proof}
(a). Assume that $(X,Y)$ has the BPBp for every Banach space $Y$,
but no such universal function $\eta_X$ exists. Then, for some
$\eps>0$, there exists a sequence of Banach spaces $\{Y_n\}$ such
that $(X,Y_n)$ has the BPBp and $\eta(X,Y_n)(\eps)\longrightarrow 0$
when $n\to \infty$. But if we consider the space
$Y=\big[\bigoplus_{n\in \N}Y_j\big]_{c_0}$, then $(X,Y)$ has the
BPBp by the BPB universality of $X$, and
Theorem~\ref{th:main-sums-together} (actually Proposition~\ref{fam}
below) gives $\eta(X,Y_n)(\eps)\geq \eta(X,Y)(\eps)>0$, a
contradiction.

(b). The same idea as in (a) works, with $\ell_1$-sums of domain
spaces substituting $c_0-$sums of range spaces.
\end{proof}

We are now ready to provide the following three propositions which will give
Theorem~\ref{th:main-sums-together}. The first is the most natural
case: $\ell_1$-sums of domain spaces and $c_0$- or
$\ell_\infty$-sums of range spaces.

\begin{proposition}\label{fam}
Let $\{X_i\, :\, i \in I\}$ and $\{Y_j\, :\, j\in J\}$ be families
of Banach spaces, $X=\big[\bigoplus_{i\in I}X_i\big]_{\ell_1}$ and
$Y=\big[\bigoplus_{j\in J}Y_j\big]_{\ell_{\infty}}$ or
$Y=\big[\bigoplus_{j\in J}Y_j\big]_{c_0}$. If the pair $(X,Y)$ has
the BPBp with $\eta(\eps)$, then $(X_i,Y_j)$ also has the
BPBp with $\eta(\eps)$ for every $i\in I$, $j\in J$. In other words,
$$
\eta(X,Y)(\eps)\leq \eta(X_i,Y_j)(\eps) \qquad \bigl(i\in I,\ j\in
J\bigr).
$$
\end{proposition}

We observe that this result appeared in
\cite[Proposition~2.4]{Choi-Kim-JFA}, but the authors were not
interested in controlling the function $\eta(\eps)$, which will be
of relevance to us in this paper.

\begin{proof}[Proof of Proposition~\ref{fam}]
Let $E_i$ and $F_j$ denote the natural isometric embeddings of $X_i$
and $Y_j$ into $X$ and $Y$, respectively, and let $P_i$ and $Q_j$
denote the natural norm-one projections from $X$ and $Y$ onto $X_i$
and $Y_j$, respectively. For $T\in L(X,Y)$, we can easily see that
$$
\|T\|=\sup\{\|Q_jT\| \,:\, j\in J\}=\sup\{\|TE_i\| \,:\, i \in I\}.
$$
(A proof of this fact can be found in \cite[Lemma~2]{PaySal}, for
instance.) Hence
$$
\|T\|=\sup\{\|Q_jTE_i\| : i\in I,~ j\in J\}  .
$$

Fix $h\in I$ and $k\in J$. To show that the pair $(X_h, Y_k)$
satisfies the BPBp with function $\eta(\eps)$, suppose that
$\|T(x_h)\|>1-\eta(\eps)$ for $T\in S_{L(X_h, Y_k)}$ and $x_h\in
S_{X_h}$. Consider the linear operator $\widetilde{T}=F_kTP_h \in
L(X,Y)$. Note that $Q_j\widetilde{T}=0$ for $ j\neq k$ and
$Q_k\widetilde{T}E_i=0$ for $i\neq h$, while
$Q_k\widetilde{T}E_h=T$, $\|\widetilde{T}\|=\|T\|=1$ and
$\|\widetilde{T}(E_h x_h)\| >1-\eta(\eps)$.

Since the pair $(X,Y)$ has the BPBp, there exists
$(x_0,\widetilde{S})\in \Pi(X,Y)$ such that
$$
\|\widetilde{T}-\widetilde{S}\|<\eps \quad \text{ and } \quad
\|x_0-E_h(x_h)\|<\eps.
$$
Let $S=Q_k\widetilde{S}E_h\in L(X_h,Y_k)$. Clearly $\|S\|\leq 1$ and
$\|S-T\|\leq\|\widetilde{S}-\widetilde{T}\|<\eps$. Now we want to
show that $S$ attains its norm at $P_h(x_0)$ and that
$\|P_h(x_0)\|=1$. Indeed, for $j\in J,~ j\neq k $ one has
$$
\|Q_j\widetilde{S}x_0\|=\|Q_j\widetilde{S}x_0-Q_j\widetilde{T}x_0\|
\leq\|\widetilde{S} - \widetilde{T}\|<\eps<1.
$$
Hence $\|\widetilde{S}x_0\|=1 =\|Q_k\widetilde{S}x_0\|$, which shows
that $\|Q_k\widetilde{S}\|=1$ and $Q_k\widetilde{S}$ attains its
norm at $x_0$. Similarly, for $i\in I $, $i\neq h$ we have
$$
\|Q_k\widetilde{S}E_i\|=\|Q_k\widetilde{S}E_i-Q_k\widetilde{T}E_i\|
\leq\|\widetilde{S}-\widetilde{T}\|<\eps<1,
$$
and so $\|Q_k\widetilde{S}\|= 1 =\|Q_k\widetilde{S}E_h\|=\|S\|$.
Since
\begin{eqnarray*} 1 &= &
\|Q_k\widetilde{S}x_0\|\leq \sum_{i\in I}\|Q_k\widetilde{S}E_iP_ix_0\|= \|SP_hx_0\|+\sum_{i\in I, i\neq h}\|Q_k\widetilde{S}E_iP_ix_0\|\\
& \leq &\|P_hx_0\|+\eps\sum_{i\in I, i\neq
h}\|P_ix_0\|\leq\|P_hx_0\|+\sum_{i\in I, i\neq h}\|P_ix_0\|=1,
\end{eqnarray*}
we have $\|P_ix_0\|=0$ for $i \neq h$ and
$\|S(P_hx_0)\|=\|P_hx_0\|=1$. Further,
\begin{equation*}
\|P_hx_0-x_h\|=\|P_h(x_0-E_h(x_h))\|\leq\|x_0-E_h(x_h)\|<\eps.\qedhere
\end{equation*}
\end{proof}

The next proposition gives when the converse result is possible, but
only for range spaces.

\begin{proposition}\label{prop-c_0sumUniversalFunction}
Let $X$ be a Banach space and let $\{Y_j\, :\, j\in J\}$ be a family
of Banach spaces. Then, for both $Y=\big[\bigoplus_{j\in
J}Y_j\big]_{c_0}$ and $Y=\big[\bigoplus_{j\in
J}Y_j\big]_{\ell_{\infty}}$, one has
$$
\eta(X,Y) = \inf\limits_{j\in J} \eta(X,Y_j).
$$
Consequently, the following four conditions are equivalent:
\begin{enumerate}
\item[(i)] $\inf\limits_{j\in J} \eta(X,Y_j)(\eps)>0$ for all $\eps\in(0,1)$,
\item[(ii)] every pair $(X, Y_j)$ has the BPBp with a common function $\eta(\eps)>0$,
\item[(iii)] the pair $\left(X, \big[\bigoplus_{j\in J}Y_j\big]_{\ell_{\infty}}\right)$ has the BPBp,
\item[(iv)] the pair $\left(X, \big[\bigoplus_{j\in J}Y_j\big]_{c_0}\right)$ has the BPBp.
\end{enumerate}
\end{proposition}

\begin{proof} The inequality $\leq$ follows
from Proposition~\ref{fam}. So let us prove the converse inequality.
To do this, we fix $\eps\in (0,1)$. Write $\eta(\eps)=\inf_{j\in J} \eta(X,Y_j)$ and suppose $\eta(\eps)>0$ (otherwise there is nothing
to prove), so $\eta(X, Y_j)\geq \eta(\eps)>0$. Let $F_j$ and $Q_j$ be as in the proof of
Proposition~\ref{fam}. Suppose that $T \in S_{L(X, Y)}$ and $x_0 \in
S_{X}$ satisfy
$$
\|Tx_0\| > 1-\eta(\eps).
$$
Then there is $j_0\in I$ such that $\|Q_{j_0}Tx_0\|>1-\eta(\eps)$.
By the assumption on $\eps$, there are an operator $S_{j_0} : X \lra
Y_{j_{0}}$ and a vector $u\in S_{X}$ such that
$$
\|S_{j_0}\|=\|S_{j_0}u\|=1,\quad \|S_{j_0}-Q_{j_0}T\|<\eps, \quad
\text{and}\quad \|x_0-u\|<\eps.
$$
Define $S : X \lra Y$ by
$$
S=\sum_{j\neq j_0} F_jQ_j T + F_{j_0}S_{j_0}.
$$
Clearly $\|S\| \leq 1$ and $\|Su\|\geq \|S_{j_0}u\|=1$, so $(u,S)\in
\Pi(X,Y)$. On the other hand,
\begin{equation*}
\|T-S\|=\sup_{j\in J} \|Q_j(T-S)\|=\|Q_{j_0}T - S_{j_0}\|<\eps.
\end{equation*}
We have proved that $(X,Y)$ has the BPBp with the function
$\eta(\eps)$. In other words, $\eta(X,Y)(\eps)\geq \eta(\eps)$.
\end{proof}

One particular case of
Proposition~\ref{prop-c_0sumUniversalFunction} is that the BPBp is
stable for finite $\ell_\infty$-sums of range spaces.

\begin{corollary}\label{cor-finite-ellinftysums}
Let $X$, $Y_1,\ldots,Y_m$ be Banach spaces and write
$Y=Y_1\oplus_\infty \cdots \oplus_\infty Y_m$. Then, $(X,Y)$ has the
BPBp if and only if $(X,Y_j)$ has the BPBp for every $j=1,\ldots,m$.
As a consequence, $Y_1,\ldots,Y_m$ are universal BPB range spaces if
and only if $Y$ is.
\end{corollary}

Note that since $\ell_1$ is not a universal BPB domain
space, we cannot expect an analogue of
Proposition~\ref{prop-c_0sumUniversalFunction} for $\ell_1$-sums of
domain spaces. Indeed, even
Corollary~\ref{cor-finite-ellinftysums} has no counterpart for
finite $\ell_1$-sums of domain spaces. This follows from the
fact that $\ell_1^2$ fails to be a universal BPB domain space
(Corollary~\ref{cor:ell12-strictlyconvex} below or
\cite[Corollary~9]{KimLee}).

The second result deals with $c_0$- or $\ell_\infty$-sums of domain
spaces.

\begin{proposition}\label{fam2}
Let $\{X_i\, :\, i \in I\}$ be a family of Banach spaces,
$X=\big[\bigoplus_{i\in I}X_i\big]_{c_0}$ or $X=\big[\bigoplus_{i\in
I}X_i\big]_{\ell_\infty}$, and let $Y$ be a Banach space. If the
pair $(X,Y)$ has the BPBp with $\eta(\eps)$, then the pair $(X_i,Y)$ also has the BPBp with $\eta(\eps)$ for every $i\in I$. In other words, $\eta(X,Y)\leq \eta(X_i,Y)$ for every $i\in I$.
\end{proposition}

\begin{proof} For every $i\in I$ we may write $X=X_i\oplus_\infty Z$ for a suitable Banach space $Z.$ Thus, without loss of generality, we may assume that $X=X_1\oplus_\infty X_2$.

For fixed $\eps\in (0,1),$ suppose that $\|T(x_0)\|>1-\eta(\eps)$
for some $T\in L(X_1, Y)$ with $\|T\|=1$ and some $x_0\in S_{X_1}$.
Define a linear operator $\widetilde{T}\in L(X,Y)$ by
$$
\widetilde{T}(x_1,x_2)=Tx_1 \qquad \bigl((x_1,x_2)\in X\bigr).
$$
Observe that $\|\widetilde{T}\|=1$ and $\|\widetilde{T}(x_0,0)\|
>1- \eta(\eps)$. Since the pair $(X,Y)$ has the BPBp, there exist $\widetilde{S}\in L(X,Y)$
with $\|\widetilde{S}\|=1$ and $(x_1',x_2')\in S_X$ such that
$$
\|\widetilde{S}\|=\|\widetilde{S}(x_1',x_2')\|=1,\quad  \|\widetilde{T}-\widetilde{S}\|<\eps\quad \text{ and} \quad \|(x_1',x_2')-(x_0,0)\|<\eps.
$$
From the last inequality, we see that
\begin{equation}\label{eq-linfty-domain-1}
\|x_1'-x_0\|<\eps, \quad  \|x_2'\|<\eps<1\quad \text{ and} \quad \|x_1'\|=1.
\end{equation}
If we define $S\in L(X_1,Y)$ by $S(x_1)=\widetilde{S}(x_1,0)$, then  $\|S\|\leq \|\widetilde{S}\|=1$
and $\|S-T\|\leq \|\widetilde{T}-\widetilde{S}\|<\eps$. So, by \eqref{eq-linfty-domain-1},
it suffices to show that  $\|S(x_1')\|=1$. Indeed, using \eqref{eq-linfty-domain-1} again,
we have that $\|\eps^{-1} x_2'\|\leq 1$, so $(x_1',\eps^{-1} x_2')\in B_X$. We can write
$$
\widetilde{S}(x_1',x_2')=(1-\eps)\widetilde{S}(x_1',0) + \eps \widetilde{S}(x_1',\eps^{-1}x_2')
$$
and, since $\|\widetilde{S}(x_1',x_2')\|=1$, we get
$\|\widetilde{S}(x_1',0)\|=1$. This is exactly $\|S(x_1')\|=1$, as desired.
\end{proof}

Considered as a real Banach space, $\ell_\infty^2$ is not a universal BPB domain space (it is isometric to $\ell_1^2$).
As a consequence, we cannot expect the converse of Proposition~\ref{fam2} to be true, even for finite sums.

The third proposition deals with the remaining case of $\ell_1$-sums of range spaces.

\begin{proposition}\label{fam3}
Let $\{Y_j\, :\, j \in J\}$ be a family of Banach spaces, $Y=\big[\bigoplus_{j\in J}Y_j\big]_{\ell_1}$, and
let $X$ be a Banach space. If the pair $(X,Y)$ has the BPBp with $\eta(\eps)$, then the pair $(X,Y_j)$ also has the BPBp with $\eta(\eps)$ for every $j\in J$. In other words, $\eta(X,Y)\leq \eta(X,Y_j)$ for every $j\in J$.
\end{proposition}

\begin{proof} Arguing as in the previous proof, we may assume that $Y=Y_1\oplus_1 Y_2$.

For $\eps\in (0,1)$ fixed, suppose that $\|T(x_1)\|>1-\eta(\eps)$ for some $T\in L(X, Y_1)$ with $\|T\|=1$
and some $x_1\in S_{X}$. Define a linear operator $\widetilde{T}\in L(X,Y)$ by
$$
\widetilde{T}(x)=(Tx,0) \qquad (x\in X).
$$
Observe that $\|\widetilde{T}\|=1$ and
$\|\widetilde{T}(x_1)\| >1-\eta(\eps)$. Since the pair $(X,Y)$ has the BPBp, there exist $\widetilde{S}\in L(X,Y)$
with $\|\widetilde{S}\|=1$ and $x_0\in S_X$ such that
$$
\|\widetilde{S}\|=\|\widetilde{S}(x_0)\|=1,\quad  \|\widetilde{T}-\widetilde{S}\|<\eps\quad \text{ and} \quad \|x_1-x_0\|<\eps.
$$
Write $\widetilde{S}(x)=(S_1(x),S_2(x))$ for every $x\in X$, where $S_j\in L(X,Y_j)$ for $j=1,2$, and observe that
\begin{equation}\label{eq:l1sumrange-1}
\|\widetilde{S}x - \widetilde{T}x\|=\|S_1 x - Tx\| + \|S_2 x\|<\eps \qquad \bigl(x\in B_X\bigr).
\end{equation}
In particular,
$$
\|S_1 - T\|<\eps \quad \text{ and } \quad \|S_2\|<\eps.
$$
Now, consider $y^*=(y_1^*,y_2^*)\in Y^*\equiv Y_1^*\oplus_\infty Y_2^*$ with $\|y^*\|=1$ such
that $\re y^*(\widetilde{S}(x_0))=1$. 
We have that
$$
1=\re y^*\bigl(\widetilde{S}(x_0)\bigr)=\re y^*_1(S_1 x_0) + \re y_2^*(S_2 x_0) \leq \|S_1 x_0\| + \|S_2 x_0\| = \|\widetilde{S}(x_0)\|=1.
$$
Therefore, we get that
$$
\re y^*_1(S_1 x_0)=\|S_1 x_0\| \quad \text{ and } \quad \re y^*_2(S_2 x_0)=\|S_2 x_0\|.
$$
Finally, we define $S\in L(X,Y_1)$ by
$$
Sx=S_1 x + y_2^*(S_2 x) \frac{S_1 x_0}{\|S_1 x_0\|} \qquad \bigl(x\in X\bigr).
$$
(Observe that $\|S_1 x_0\|=1-\|S_2 x_0\|\geq 1 - \|S_2\|>1-\eps>0$). Then, for $x\in B_X$ we have
$$
\|S x\|\leq \|S_1 x\| + |y_2^*(S_2 x)|\leq \|S_1 x\| + \|S_2 x\|=\|\widetilde{S}(x)\|\leq 1,
$$
so $\|S\|\leq 1$. Furthermore,
\begin{align*}
\|S x_0\| & =\left\| S_1 x_0 + y_2^*(S_2 x_0)
\frac{S_1 x_0}{\|S_1 x_0\|}\right\| \geq
\re y_1^*\left(S_1 x_0 + y_2^*(S_2 x_0)\frac{S_1 x_0}
{\|S_1 x_0\|}\right) \\ & = \re y_1^*(S_1 x_0) + \re y_2^*(S_2 x_0) =1.
\end{align*}
Hence $S$ attains its norm at $x_0$. As $\|x_0-x_1\|<\eps$, it remains to prove that $\|S-T\|<\eps$. Indeed, for $x\in B_X$, we have
$$
\|S x-Tx\|\leq \|S_1 x - T x\| + |y_2^*(S_2x)|\leq \|S_1 x - T x\| + \|S_2 x\|,
$$
so $\|S-T\|<\eps$ by \eqref{eq:l1sumrange-1}.
\end{proof}

As $\ell_1$ does not have Lindenstrauss property B, we cannot expect that a converse of Proposition~\ref{fam3}
can be true in general. In fact, we don't even know whether such a converse is true for finite sums, that is,
whether $(X,Y_1\oplus_1 Y_2)$ has the BPBp whenever $(X,Y_j)$ does for $j=1,2$.

Another result in the same direction is the following generalization of \cite[Theorem~11]{ChoiKimSK} where it is proved for $X=\ell_1$.

\begin{prop}\label{prop-C(K,Y)}
Let $X$ and $Y$ be Banach spaces and let $K$ be a compact Hausdorff space. If $(X,C(K,Y))$ has the BPBp
with a function $\eta(\eps)$, then $(X,Y)$ has the BPBp with the same function $\eta(\eps)$. In other words, $\eta(X,Y)\geq \eta\bigl(X,C(K,Y)\bigr)$.
\end{prop}

\begin{proof}
Given $\eps>0$, consider $T\in S_{L(X,Y)}$ and $x_0\in S_X$ satisfying
$$
\|Tx_0\|>1-\eta(\eps).
$$
The bounded linear operator $\widetilde{T}:X\longrightarrow C(K,Y)$, defined by $[\widetilde{T}(x)](t) = T(x)$
for all $x\in X$ and $t\in K$, satisfies $\|\widetilde{T} x_0\|= \|T x_0\|>1-\eta(\eps)$.
By the assumption, there exist $x_1\in S_X$ and $\widetilde{S}\in L(X, C(K,Y))$ such that
$$
\|\widetilde{S}\|=\|\widetilde{S}(x_1)\|=1,\quad \|\widetilde{T} - \widetilde{S}\|<\eps, \quad \|x_0 - x_1\|<\eps.
$$
Moreover, there is $t_1 \in K$ such that $1=\|\widetilde{S}(x_1)\|=\bigl\|[\widetilde{S}(x_1)](t_1)\bigr\|$.
We can now see that the bounded linear operator $S:X\longrightarrow Y$ defined by $S(x) = [\widetilde{S}(x)](t_1)$ for all $x\in X$, satisfies
that
$$
\|S\|= \sup_{x\in B_X} \bigl\|[\widetilde{S}(x)](t_1)\bigr\| = \bigl\|[\widetilde{S}(x_1)](t_1)\bigr\|=\|S(x_1)\|=1
$$
and
\begin{align*}
\|S-T\| &= \sup_{x\in B_X} \|S(x) - T(x)\| = \sup_{x\in B_X} \bigl\|[\widetilde{S}(x)](t_1) -
[\widetilde{T}(x)](t_1)\bigr\| \\ &\leq \sup_{x\in B_X} \bigl\|\widetilde{S}(x) - \widetilde{T}(x)\bigr\|= \|\widetilde{S} - \widetilde{T}\|<\eps.
\end{align*}
As we have already known that
$\|x_1 - x_0\|<\eps$, we obtain that $(X,Y)$ has the BPBp with the function $\eta(\eps)$.
\end{proof}

Let us observe that the converse implication in the above proposition is false, as $Y=\K$ is a universal BPBp range space
by the Bishop-Phelps-Bollob\'{a}s theorem but, $C[0,1]$ does not have even Lindenstrauss property B \cite{Schacher-Pacific}.

A more general way of stating Proposition~\ref{prop-C(K,Y)} is that if $(X,C(K) \hat{\otimes}_\epsilon Y)$
has the BPBp with $\eta(\epsilon),$ then so does $(X,Y)$. We do not know what other spaces, besides $C(K),$ have this property.
In an analogous way, noting
that $L_1(\mu,X) = L_1(\mu) \hat{\otimes}_\pi X,$ we remark that we  do not know if a result similar
to Proposition~\ref{prop-C(K,Y)} can be obtained for vector valued $L_1$-spaces
in the domain: that is, we do not know if the fact that $(L_1(\mu,X),Y)$ has the BPBp implies that $(X,Y)$ does
as well (where $X$ and $Y$ are Banach spaces and $\mu$ is a positive measure).

We finish this section with a discussion on some analogues to Propositions \ref{fam}, \ref{fam2}, and \ref{fam3} for norm-attaining operators.
Let $X$, $Y$, $X_1$, $X_2$, $Y_1$, and $Y_2$ be Banach spaces. It is proved in \cite[Lemma~2]{PaySal} that
Proposition~\ref{fam} has a counterpart for norm attaining operators; that is, if $NA(X_1\oplus_1 X_2,Y_1\oplus_\infty Y_2)$
is dense in $L(X_1\oplus_1 X_2,Y_1\oplus_\infty Y_2)$, then $NA(X_i,Y_j)$ is dense in $NA(X_i,Y_j)$ for $i,j\in\{1,2\}$.
If we consider $\ell_1$-sums in the range space, it is possible to adapt the proof of Proposition~\ref{fam3} to this case, obtaining
that if $NA(X, Y_1\oplus_1 Y_2)$ is dense in $L(X,Y_1\oplus_1 Y_2)$, then $NA(X,Y_j)$ is dense in
$L(X,Y_j)$ for $j=1,2$. As far as we know, this result is new. On the other hand, if we consider $\ell_\infty$-sums of domain spaces,
we do not know how to adapt the proof of Proposition~\ref{fam2} to the case of norm-attaining operators. We do not know if the result
is true or not; that is, \ we do not know whether the fact that $NA(X_1\oplus_\infty
X_2,Y)$ is dense in $L(X_1\oplus_\infty X_2,Y)$ forces $NA(X_i,Y)$
to be dense in $L(X_i,Y)$, $i=1,2$.

Let us summarize this new result here (actually, we state a formally more general result) and deduce from it what we believe is an interesting consequence.

\begin{prop}
Let $\{Y_j\, :\, j \in J\}$ be a family of Banach spaces,
$Y=\big[\bigoplus_{j\in J}Y_j\big]_{\ell_1}$, and let $X$ be a
Banach space. If $NA(X, Y)$ is dense in $L(X,Y)$, then $NA(X,Y_j)$
is dense in $L(X,Y_j)$ for every $j\in J$.
\end{prop}

\begin{example}\label{example-C01-L101}
{\slshape Consider the Banach space $X=C[0,1]\oplus_1 L_1[0,1]$.
Then, $NA(X,X)$ is not dense in $L(X,X)$.\ } Indeed, if $NA(X,X)$
were dense in $L(X,X)$, then $NA(L_1[0,1],X)$ would be dense in
$L(L_1[0,1],X)$ by \cite[Lemma~2]{PaySal}. But the above proposition
would imply that $NA(L_1[0,1],C[0,1])$ is dense in
$L\bigl(L_1[0,1],C[0,1]\bigr)$, a result which was proved to be
false by W.~Schachermayer \cite{Schacher-Pacific}.
\end{example}

We will provide another example of this in section~\ref{sec:range}.

\section{Results on domain spaces}\label{sec:domain}

We have two objectives in this section. First, we will show that every Banach space of dimension
greater than one can be renormed to not be a universal BPBp domain space. One should compare this result with
the result by J.~Bourgain \cite{Bou} asserting that a Banach space has Lindenstrauss property A
in every equivalent norm if and only if the space has the Radon-Nikod\'{y}m property. Later in this section
we will study conditions which ensure that a Banach space is a universal BPB domain space.

\begin{theorem}\label{th:BPBp-A-renorming}
The base field $\K=\R~\mbox{or}~\C$ is the unique Banach space which is a universal BPBp domain space in any equivalent renorming.
\end{theorem}

The theorem is a direct result of the following lemma.

\begin{lemma}\label{lemma:L-summand-strictlyconvex}
Let $X$ be a Banach space containing a non-trivial $L$-summand (i.e.\ $X=X_1\oplus_1 X_2$ for some non-trivial
subspaces $X_1$ and $X_2$) and let $Y$ be a strictly convex Banach space. If the pair $(X,Y)$ has the BPBp, then
$Y$ is uniformly convex.
\end{lemma}

\begin{proof}
For $j = 1, 2,$ we pick $e_j\in S_{X_j}$ and $e_j^*\in S_{X^*}$ such that $e_j^*(e_j)=1$ and
$e_1^*(X_2)=0$ and $e_2^*(X_1)=0$ (just identify $X_j^*$ with a subspace of $X^*$ and extend the functional to be
zero on the other subspace).

Fix $\eps\in (0,1/2)$ and recall that $\eta(X,Y)(\eps)>0$ by the BPBp. Consider $y_1,y_2\in S_Y$ such that
$$
\|y_1+y_2\|> 2 - 2 \eta(X,Y)(\eps).
$$
We define an operator $T\in L(X,Y)$ by
$$
T(x_1,x_2)=e_1^*(x_1)y_1 \, + \, e_2^*(x_2)y_2 \qquad \bigl((x_1,x_2)\in X\bigr),
$$
which satisfies $\|T\|=1$. As
$$
\left\|T\bigl(\tfrac12 e_1,\tfrac12 e_2\bigr)\right\| = \tfrac12 \|y_1 + y_2 \|> 1 - \eta(X,Y)(\eps),
$$
there are $(x_1,x_2)\in S_X$ and $S\in L(X,Y)$ such that
$$
\|S\|=\|S(x_1,x_2)\|=1,\quad \left\|\tfrac12e_1-x_1\right\| + \left\|\tfrac12e_2-x_2\right\|<\eps \quad \text{ and } \quad \|T-S\|<\eps.
$$
We deduce that $\left|\tfrac12 - \|x_j\| \right|<\eps<1/2$, so $0 < \|x_j\| < 1$ and
$$
\left\|e_j-\frac{x_j}{\|x_j\|}\right\| \leq 2 \left\|\tfrac12 e_j - x_j\right\| + \left\|2x_j - \frac{x_j}{\|x_j\|}\right\|\leq 2\eps + \bigl|1-2\|x_j\| \bigr| <4\eps
$$
for $j = 1, 2$. If we write
$$
z_1=S\left(\frac{x_1}{\|x_1\|},0\right)\in B_{Y} \quad \text{ and } \quad z_2=S\left(0,\frac{x_2}{\|x_2\|}\right)\in B_{Y},
$$
we have that
$$
1=\|S(x_1,x_2)\|=\bigl\| \|x_1\| z_1 + \|x_2\| z_2\bigr\|\quad  \text{ and } \quad \|x_1\| + \|x_2\|=1.
$$
As $Y$ is strictly convex, it follows that $z_1=z_2$. Now,
\begin{align*}
\|y_1 - y_2\| & = \|T(e_1,0) - T(0,e_2)\| \\ & \leq \bigl\|T(e_1,0) - S(e_1,0)\bigr\| + \bigl\|T(0,e_2)-S(0,e_2)\bigr\| \\ & \qquad\qquad  +
\bigl\|S(e_1,0) - z_1\bigr\| + \bigl\|S(0,e_2)-z_2\bigr\| \\ &< 2\|T-S\| + \left\|S(e_1,0) - S\left(\frac{x_1}{\|x_1\|},0\right)\right\| +
\left\|S(0,e_2) - S\left(0,\frac{x_2}{\|x_2\|}\right)\right\| \\ &\leq 2\eps + \|S\|\left(\left\|e_1-\frac{x_1}{\|x_1\|}\right\| + \left\|e_2-\frac{x_2}{\|x_2\|}\right\|\right)<10\eps.
\end{align*}
This implies that $Y$ is uniformly convex, as desired.
\end{proof}

\begin{proof}[Proof of Theorem~\ref{th:BPBp-A-renorming}]
If $\dim(X)>1$, we consider any one-codimensional subspace $Z$ of $X$ and observe that
$X\simeq \widetilde{X}:=\K \oplus_1 Z$. Now, consider any strictly convex Banach space $Y$
which is not uniformly convex, and the above lemma gives that $(\widetilde{X},Y)$ does not have the BPBp.
\end{proof}

We next give the following particular case of Lemma~\ref{lemma:L-summand-strictlyconvex}, which can be deduced
from arguments given in \cite{AAGM2}.

\begin{corollary}\label{cor:ell12-strictlyconvex}
Let $Y$ be a strictly convex Banach space. If $(\ell_1^2,Y)$ has the BPBp, then $Y$ is uniformly convex.
\end{corollary}

A nice consequence of the above corollary is the following example.

\begin{example}
{\slshape There exists a reflexive Banach space $X$ such that the pair $(X,X)$ fails the BPBp.\ }
Indeed, let $Y$ be a reflexive strictly convex space which is not uniformly convex and consider the reflexive
space $X=\ell_1^2\oplus_1 Y$. If the pair $(X,X)$ had the BPBp, then so would $(\ell_1^2,Y)$ by
Theorem~\ref{th:main-sums-together}, a contradiction with Corollary~\ref{cor:ell12-strictlyconvex}.
\end{example}

Our next goal is to give some necessary conditions for a Banach space to be a
universal BPB domain space. Let us recall that if a Banach space $X$ has Lindenstrauss property A,
then the following hold  \cite{Lindens}: (1) if $X$ is isomorphic to a strictly convex space, then $S_X$ is
the closed convex hull of its extreme points, and (2) if $X$ is isomorphic to a locally uniformly convex
space, then $S_X$ is the closed convex hull of its strongly exposed points. These results have been
strengthened to the case of universal BPBp domain spaces in \cite[Theorem~8 and Corollary~9]{KimLee} but with
the additional hypothesis that there is a common function $\eta$ giving the BPBp for all range spaces.
Thanks to Corollary~\ref{corollary:universal-eta}, this hypothesis is unnecessary.

\begin{corollary}\label{tool123} Let $X$ be a universal BPB domain space. Then,
\begin{enumerate}
\item[(a)] in the real case, there is no face of $S_X$ which contains a non-empty relatively open subset of $S_X$;
\item[(b)] if $X$ is isomorphic to a strictly convex Banach space, then the set of all extreme points of $B_X$ is dense in $S_X$;
\item[(c)] if $X$ is superreflexive, then the set of all strongly exposed points of $B_X$ is dense in $S_X$.
\end{enumerate}
In particular, if $X$ is a real $2$-dimensional Banach space which is a universal BPB domain space, then $X$ is uniformly convex.
\end{corollary}

We don't know if a universal BPB domain space has to be uniformly convex. Nevertheless, we can improve Corollary~\ref{tool123}.(c)
to get a slightly stronger result. To do this, we follow Lindenstrauss \cite{Lindens} to say that a family $\left\{x_\alpha\right\}_\alpha
\subset S_X$ is \emph{uniformly strongly exposed} (with respect to a family $\left\{f_\alpha\right\}_\alpha\subset S_{X^*}$) if there
is a function $\eps \in (0,1) \longmapsto \delta(\eps)>0$ having the following properties:\\
\indent (i) $f_\alpha(x_\alpha) = 1$ for every $\alpha$, and\\
\indent (ii) for any $x\in B_X$, $\re f_\alpha(x)> 1-\delta(\eps)$ implies $\|x-x_\alpha\|<\eps$.\\
In a uniformly convex Banach space,
the unit sphere is a uniformly strongly exposed family (actually, this property characterizes uniform convexity). Here is our new result.

\begin{corollary}
Let $X$ be a superreflexive universal BPB domain space. Then for every $\eps_0\in (0,1)$ there exists
an $\eps_0$-dense uniformly strongly exposed family. In particular, the set of all strongly exposed points of $B_X$ is dense in $S_X$.
\end{corollary}

\begin{proof} We write $\|\cdot\|$ for the given norm of $X$ and consider an equivalent norm $|||\cdot |||$ on $X$ for which the Banach space $\left(X, |||\cdot |||\right)$
is uniformly convex. We may assume that $|||x|||\leq \|x\|$ for all $x\in X$ (see \cite[Theorem~9.14]{FHHMZ},
for instance). For each $m\in \mathbb{N}$, we define an equivalent norm on $X$ by
$$
\|x\|_m := \left(\|x\|^2 + \tfrac1m|||x|||^2 \right)^{1/2}\qquad \bigl(x\in X\bigr).
$$
We observe that
$X_m=\left(X, \|\cdot \|_m\right)$ is uniformly convex, so we may consider the function
$\eps\longmapsto \delta_m(\eps)$ (from the uniform convexity of $X_m$) which gives that $S_{X_m}$ is a uniformly
strongly exposed family (with respect to $S_{X_m^*}$). That is, for every $y\in S_{X_m}$ there is $y^*\in S_{X_m^*}$ such that $y^*(y)=1$ and
\begin{equation}\label{eq:uniformlyconvex-m}
{\rm if\ } z\in B_{X_m}\ {\rm is \ such \ that \ } \ \re y^*(z)>1-\delta_m(\eps), {\rm \ then \ }  \|z-y\|_m<\eps.
\end{equation}
Let $I_m: (X,\|\cdot\|)\longrightarrow (X_m,\|\cdot\|_m)$ be the formal identity operator and write
$$
T_m = \frac{I_m}{\|I_m\|} \quad \text{ and } \quad a_m=\frac{\sqrt{m+1}}{\sqrt{m}}
$$
for every $m\in \N$. Since
\begin{equation*}
1\leq \|I_m\| \leq a_m \quad \text{ and } \quad a_m^{-1} \leq \|I_m^{-1}\|\leq 1,
\end{equation*}
it follows that
$$
a_m^{-1}\leq \|T_m^{-1}\|\leq a_m.
$$
Therefore, if $S\in L(X,X_m)$ satisfies $\|T_m-S\|<a_m^{-1}\leq \|T_m^{-1}\|^{-1}$, then $S$ is invertible (see \cite[Corollary~18.12]{Jameson}) and
$$
\|T_m^{-1}-S^{-1}\|\leq \frac{\|T_m^{-1}\|^2\,\|T_m-S\|}{1-\|T_m^{-1}\|\,\|T_m-S\|} \leq \frac{a_m^2\|T_m-S\|}{1-a_m\|T_m-S\|},
$$
(see \cite[Lemma~15.11]{FHHMZ} where it is done for the case $X_m=X$, but the general case easily follows). Then
\begin{equation}\label{eq:unifmlyconvex-S-1}
\|S^{-1}\|\leq \|T_m^{-1}\|+\|T_m^{-1}-S^{-1}\|\leq a_m + \frac{a_m^2\|T_m-S\|}{1-a_m\|T_m-S\|}=\frac{a_m}{1-a_m\|T_m-S\|}.
\end{equation}
Finally, let $\eps\longmapsto \eta_X(\eps)>0$ be the universal BPBp function for $X$ given by Corollary~\ref{corollary:universal-eta}.

We need to show that for every fixed $\eps_0\in (0,1),$  there exists a function $\eps\longmapsto \delta_{\eps_0}(\eps)>0$ satisfying
the following conditions: For each $x_0\in S_X$ there exist $x_1\in S_X$ and $x_1^*\in S_{X^*}$ such that\\
\indent (i) $\|x_1-x_0\|<\eps_0$,\\
\indent (ii)  $x_1^*(x_1)=1$, and\\
\indent (iii) for every $\eps>0$, if $x\in B_X$ satisfies $\re x_1^*(x)>1-\delta_{\eps_0}(\eps)$, then $\|x_1-x\|<\eps$.

Indeed, fix $\eps_0\in (0,1)$ and choose $m\in \N$ satisfying
$$
a_m^{-1}=\frac{\sqrt m}{\sqrt{1+m}}>1-\eta_X(\eps_0) \quad \text{ and } \quad a_m^{-1}>\eps_0.
$$
Consider $\delta_{\eps_0}(\eps)= \delta_m\left(\dfrac{1-a_m\eps_0}{a_m} \,\eps\right)$ ($\eps>0$) for
this $m$. Observe that the operator $T_m\in L(X,X_m)$ satisfies $\|T_m\|=1$ and
\[
\|T_m x_0 \|_m = \frac{\|x_0\|_m}{\|I_m\|} \geq \|x_0\|\,a_m^{-1}>1-\eta_X(\eps_0).
\]
Hence, by the BPB property of $(X,X_m)$, there exist both an operator $S\in L(X,X_m)$ with $\|S\|=1$ and $x_1\in S_X$ such that $$
\|x_0 - x_1 \| < \eps_0, \quad \|S(x_1)\|_m=1, \quad \text{and} \quad \|T_m-S\|<\eps_0.
$$
Now, for $y=S x_1\in S_{X_m}$, let $y^*\in S_{X_m^*}$ be the functional satisfying \eqref{eq:uniformlyconvex-m}
(i.e.\ $y^*$ strongly exposed $S x_1$ with the function $\delta_m(\cdot)$). So $y^*\left(Sx_1\right) = 1$ and for $x\in B_X$ and $\eps'>0$,
$$
\re y^*(Sx)>1-\delta_{m}(\eps') \quad \text{implies} \quad \|S x - Sx_1\|< \eps'.
$$
Consider $x_1\in S_X$ (which satisfies $\|x_1-x_0\|<\eps_0$) and $x_1^*=S^*(y^*)\in X^*$ (which satisfies
$x_1^*(x_1)=1$ and $\|x_1^*\|=1$). Suppose that for some $\eps>0$, $x\in B_X$ satisfies $\re x_1^*(x)>1-\delta_{\eps_0}(\eps)$. Then
$$
\re y^*(Sx)=\re x_1^*(x)>1-\delta_{\eps_0}(\eps)=1-\delta_m\left(\frac{1-a_m\eps_0}{a_m} \,\eps\right),
$$
so $\|Sx - Sx_1\|< \dfrac{1-a_m\eps_0}{a_m}\,\eps$. On the other hand, as $\|T_m-S\|<\eps_0<a_m^{-1}\leq \|T_m^{-1}\|^{-1}$,
it follows that $S$ is invertible and we get from \eqref{eq:unifmlyconvex-S-1} that $\|S^{-1}\|\leq \tfrac{a_m}{1-a_m\|T_m-S\|}.$ Therefore
\begin{equation*}
\|x - x_1\|\leq \|S^{-1}\|\,\|Sx - Sx_1\| < \|S^{-1}\|\frac{1-a_m\eps_0}{a_m}\,\eps \leq
\frac{a_m}{1-a_m\|T_m-S\|}\,\frac{1-a_m\eps_0}{a_m}\,\eps < \eps.\qedhere
\end{equation*}
\end{proof}

One can ask whether it is actually possible to deduce uniform convexity from the above corollary. This is not the
case, even in the finite dimensional case. To see this, just consider a three dimensional space $X$ in which the set $M$ of strongly
exposed points is dense but not the all of $S_X.$ (For instance, we may modify the Euclidean sphere in such a
way that there are two diametrically opposite small line segments and the rest of the points are still strongly exposed.)
Now, for every $\eps_0>0$, consider the set of those points in $S_X$ whose distance to $S_X\setminus M$ is greater
than or equal to $\eps_0/2$ (which is contained in $M$). Then this set is a closed subset of $S_X$ consisting of strongly
exposed points, and so it is uniformly strongly exposed by compactness. On the other hand, it is clearly $\eps_0$-dense.

\section{Results on range spaces}\label{sec:range}

Our main goal here is to give an example of a Banach space having Lindenstrauss property $B$ which is not a universal
BPB range space. We recall that, as a particular case of \cite[Theorem~2.2]{AAGM2}, finite-dimensional real polyhedral spaces are universal BPB range spaces since they have property $\beta$ (we will not introduce the definition of property $\beta$ here since we are not going to work with it).

We are now able to present the main example of the section.

\begin{example}\label{examp}
{\slshape For $k\in \N$, consider $Y_k=\R^2$ endowed with the norm
$$
\|(x,y)\|=\max\{|x|,|y|+\tfrac{1}{k} |x|\} \quad (x,y\in \R).
$$
Observe that $B_{Y_k}$ is the absolutely convex hull of the set
$\bigl\{(0,1),~(1,1-\tfrac{1}{k}),~(-1,1-\tfrac{1}{k})\bigr\}$, so $Y_k$ is polyhedral and, therefore, it is a universal BPB range space by \cite[Theorem~2.2]{AAGM2}.
Then, we have that
$$
\inf\limits_{k\in\N} \eta(\ell_1^2,Y_k)(\eps)=0
$$
for every $\eps\in (0,1/2)$. Therefore, if we consider
$$
\mathcal{Y}=\big[\bigoplus_{i=1}^\infty Y_k\big]_{c_0}, \qquad \mathcal{Z}=\big[\bigoplus_{i=1}^\infty Y_k\big]_{\ell_1} \quad \text{ and } \quad \mathcal{W}=\big[\bigoplus_{i=1}^\infty Y_k\big]_{\ell_\infty},
$$
then none of the pairs $(\ell_1^2,\mathcal{Y})$,  $(\ell_1^2,\mathcal{Z})$ and $(\ell_1^2,\mathcal{W})$ has the BPBp.}
\end{example}

\begin{proof}
Define $T_k\in L(\ell_1^2,Y_k)$ by
$$
T_k(e_1)=(-1,1-\tfrac{1}{k}) \quad \text{ and } \quad T_k(e_2)=(1,1-\tfrac{1}{k}).
$$
Clearly $\|T_k\|=1$ and $T_k(\tfrac12 e_1+\tfrac12 e_2)=(0,1-\tfrac{1}{k})$. Hence, $\|T_k(\tfrac12 e_1+\tfrac12 e_2)\|=1-\tfrac{1}{k}$.
Assume that for some $1/2>\eps>0$ we have
$$
\inf_{k\in\N}\eta(\ell_1^2,Y_k)(\eps)>0
$$
and take $\eta(\eps)$ such that $\inf_{k\in\N}\eta(\ell_1^2,Y_k)(\eps)>\eta(\eps)>0$.
Then, for every $k\in \N$ such that $1-\tfrac{1}{k}>1-\eta(\eps)$, we can find
$S_k\in L(\ell_1^2,Y_k)$ with $\|S_k\|=1$ and $u_k\in S_{\ell_1^2}$ such that
$$\|S_ku_k\|=1,\quad \|T_k-S_k\|<\eps \quad \text{and} \quad \left\|u_k-\left(\tfrac12 e_1+ \tfrac12 e_2\right)\right\|<\eps.
$$
Now, as $\left\|u_k- \left(\tfrac12 e_1+ \tfrac12 e_2\right)\right\|<1/2$, we have that $u_k$ lies in the interior
of the interval $[e_1,e_2]\subset S_{\ell_1^2}$ and that $\|S_k(u_k)\|=1.$  It follows that the entire interval
$[S_k(e_1),S_k(e_2)]$ lies in the unit sphere of $Y_k$, and so $\|S_k(e_1)-S_k(e_2)\|\leq 1$. Since $\|T_k-S_k\|<\eps$, we get that
$$\|T_k(e_1)-S_k(e_2)\|\leq \|T_k(e_1)-S_k(e_1)\|+\|S_k(e_1)-S_k(e_2)\|<\eps+1<3/2.$$

On the other hand, since $\|T_k(e_1)-T_k(e_2)\|=2$,
$$\|T_k(e_1)-S_k(e_2)\|\geq \|T_k(e_1)-T_k(e_2)\|-\|T_k(e_2)-S_k(e_2)\|>2-\eps>3/2,$$ a contradiction.

Finally, the last assertion is a direct consequence of Propositions \ref{fam} (for the $c_0$- and $\ell_\infty$-sums)
and \ref{fam3} (for the $\ell_1$-sum).
\end{proof}

Here is the main consequence of the example above.

\begin{theorem}
Lindenstrauss property B does not imply being a universal BPB range space.
\end{theorem}

\begin{proof}
Just consider $\mathcal{Y}$ in the above example. Then $\mathcal{Y}$ has Lindenstrauss property B as a sum of
$Y_k$'s, each of which has it, and this property is stable under $c_0$-sums (see \cite[Proposition~3]{AAP-quasibeta}). However, $\mathcal Y$ is not a universal BPB range space since $(\ell_1^2,\mathcal{Y})$ does not have BPBp.
\end{proof}

In \cite{AAP-quasibeta}, a property called quasi-$\beta$ was introduced as a weakening of
property $\beta$ which still implies property B. The above argument  shows that property quasi-$\beta$ does not implies being a universal BPB range space:

\begin{corollary}
Property quasi-$\beta$ does not imply being a universal BPB range space.
\end{corollary}

\begin{proof}
As in the theorem above, consider the space $\mathcal{Y}$ of Example~\ref{examp}. Then, $\mathcal{Y}$
has property quasi-$\beta$ since it is a $c_0$-sum of spaces with property $\beta$ and we may use \cite[Proposition~4]{AAP-quasibeta}.
\end{proof}

We now list more consequences of Example~\ref{examp}, showing first that there is no
infinite counterpart to Corollary~\ref{cor-finite-ellinftysums}.

\begin{corollary}
The BPBp is not stable under infinite $c_0$- or $\ell_\infty$-sums of the range space.
Even more, being a universal BPB range space is not stable under infinite $c_0$- or $\ell_\infty$-sums.
\end{corollary}

It is shown in \cite{Kim-c_0} that if $Y$ is a real strictly convex Banach space which is not uniformly convex,
then $(c_0,Y)$ fails to have the BPBp. We are now able to show the same for some non-strictly convex spaces.

\begin{corollary} Let $\mathcal{Y}$, $\mathcal{Z}$ and $\mathcal{W}$ be the spaces of Example~\ref{examp}.
Then none of $\mathcal{Y}$, $\mathcal{Z}$ and $\mathcal{W}$ is strictly convex and, in the real case, none of
the pairs $(c_0,\mathcal{Y})$, $(c_0,\mathcal{Z})$, $(c_0,\mathcal{W})$ has the BPBp.
\end{corollary}

\begin{proof}
Notice that $c_0$ and $\ell_1^2\oplus_\infty c_0$ are isometric, so we consider $\ell_1^2\oplus_\infty c_0$.
If $(\ell_1^2\oplus_\infty c_0, \mathcal{Y})$ has the BPBp, then from Proposition~\ref{fam2}, $(\ell_1^2,\mathcal{Y})$
has the BPBp, contradicting Example~\ref{examp}. The same argument works for $(c_0,\mathcal{Z})$.
\end{proof}

Finally, we have the following negative results on the stability of the so-called AHSP. Recall that a Banach space $Y$ has the \emph{Approximate Hyperplane Series Property} (\emph{AHSP}) \cite{AAGM2} if it satisfies a geometrical condition which is equivalent to the fact that $(\ell_1,Y)$ has the BPBp.

\begin{corollary}
The AHSP is not stable under infinite $c_0$-, $\ell_1$- or $\ell_\infty$-sums.
\end{corollary}

\begin{proof} Let $\mathcal{Y}=\big[\bigoplus_{i=1}^\infty Y_k\big]_{c_0}$ be the space given in Example~\ref{examp}.
As $\ell_1\equiv \ell_1^2 \oplus_1 \ell_1$, it follows from Proposition~\ref{prop-c_0sumUniversalFunction}
that $(\ell_1,\mathcal{Y})$ does not have the BPBp, from which we deduce that $\mathcal{Y}$ does not have the
AHSP \cite{AAGM2}. On the other hand, all the $Y_k$ have the AHSP since they are finite-dimensional
\cite{AAGM2}. For the $\ell_1$-sum and the $\ell_\infty$-sum, the argument is the same considering $\mathcal{Z}=\big[\bigoplus_{i=1}^\infty Y_k\big]_{\ell_1}$ and $\mathcal{W}=\big[\bigoplus_{i=1}^\infty Y_k\big]_{\ell_\infty}$.
\end{proof}


\begin{thebibliography}{99}

\bibitem{Acosta-RACSAM}
\textsc{M.~D.~Acosta},
Denseness of norm attaining mappings, \emph{Rev. R. Acad. Cien. Serie A. Mat.} {\bf 100} (2006), 9--30.

\bibitem{AAGM2}
\textsc{M.~D.~Acosta, R.~M.~Aron, D.~Garc\'ia and M.~Maestre}, The Bishop-Phelps-Bollob\'as Theorem for
 operators, \emph{J. Funct. Anal.} {\bf 254} (2008),
  2780--2799.

\bibitem{ABGM}
\textsc{M.~D.~Acosta, J.~Becerra-Guerrero,
 D.~Garc\'ia and M.~Maestre},
The Bishop-Phelps-Bollob\'as Theorem
 for bilinear forms, \emph{Trans. Amer. Math. Soc.} {\bf 365} (2013), 5911-5932.

\bibitem{AAP-quasibeta} \textsc{M.~D.~Acosta, F.~J.~Aguirre, and
R.~Pay\'{a}}, A new sufficient condition for the denseness of norm attaining operators, \emph{Rocky Mountain J. Math.} \textbf{26} (1996), 407--418.

\bibitem{ACGM-Adv}
\textsc{R.~M.~Aron, Y.~S.~Choi,  D.~Garc\'{\i}a and  M.~Maestre},
The Bishop-Phelps-Bollob\'{a}s Theorem for ${\mathcal{L}}(L_1(\mu), L_\infty[0,1])$, \emph{Adv. Math.} {\bf 228} (2011), 617--628.

\bibitem{BishopPhelps} \textsc{E.~Bishop and R.~R.~Phelps}, A proof that every Banach space is subreflexive, \emph{Bull. Amer. Math. Soc.} {\bf 67} (1961), 97-98.

\bibitem{Bollobas} \textsc{B.~Bollob\'as}, An extension to the theorem of Bishop and Phelps, \emph{Bull. London. Math. Soc.} {\bf 2} (1970), 181-182.

\bibitem{Bou} \textsc{J.~Bourgain}, Dentability and the Bishop-Phelps property, \emph{Israel J. Math.} {\bf 28} (1977), 265-271.

\bibitem{CasGuiKad}
\textsc{B.~Cascales, A.~J.~Guirao, and V.~Kadets}, A Bishop-Phelps-Bollob\'{a}s type theorem for uniform algebras, \emph{Adv. Math.} \textbf{240} (2013), 370--382.

\bibitem{ChoiKimSK}
\textsc{Y.~S.~Choi and S.~K.~Kim}, The Bishop�-�Phelps-��Bollob\'{a}s property and lush spaces, \emph{J. Math. Anal. Appl.} \textbf{390} (2012), 549--555.

\bibitem{Choi-Kim-JFA} \textsc{Y.~S.~Choi and S.~K.~Kim}, The Bishop-��Phelps-��Bollob\'{a}s theorem for operators from $L_1(\mu)$ into a Banach space with the Radon-Nykod\'{y}m property, \emph{J. Funct. Anal.} \textbf{261} (2011), 1446--1456.


\bibitem{CS-quasialpha}
\textsc{Y.~S.~Choi and H.~G.~Song}, Property (quasi-$\alpha$) and the denseness of norm attaining mappings, \emph{Math. Nachrichten} \textbf{281} (2008), 1-9.

\bibitem{CS}
\textsc{Y.~S.~Choi and H.~G.~Song}, The Bishop-Phelps-Bollob\'as theorem fails for bilinear forms on $l_1 \times l_1$, \emph{J. Math. Anal. Appl.} {\bf 360} (2009), 752--753.

\bibitem{FHHMZ} \textsc{M.~Fabian, P.~Habala, P.~H\'{a}jek, V.~Montesinos, and V.~Zizler}, \emph{Banach space theory}, CMS Books in Mathematics, Springer, New York, 2011.

\bibitem{Jameson} \textsc{G.~J.~O.~Jameson},
\emph{Topology and Normed Spaces}, Chapman and Hall, London, 1974.

\bibitem{Kim-c_0} \textsc{S.~K.~Kim}, The Bishop-Phelps-Bollob\'as Theorem for operators from $c_0$ to uniformly convex spaces, \emph{Israel J. Math.} {\bf 197} (2013), 425-435.

\bibitem{KimLee} \textsc{S.~K.~Kim and H.~J.~Lee}, Uniform convexity and Bishop-Phelps-Bollob\'{a}s property, \emph{Canadian J. Math.} {\bf 66} (2014), 373-386.

\bibitem{Lindens} \textsc{J.~Lindenstrauss}, On operators which attain their norm, \emph{Israel J. Math.} {\bf 1} (1963), 139-148.

\bibitem{PaySal} \textsc{R.~Pay\'{a} and Y.~Saleh}, Norm attaining operators from $L_1(\mu)$ into $L_\infty(\nu)$, \emph{Arch. Math.} \textbf{75} (2000), 380--388.


\bibitem{Partington} \textsc{J.~R.~Partington}, Norm attaining operators, \emph{Israel J. Math.} \textbf{43} (1982), 273--276.

\bibitem{Schacher}
\textsc{W.~Schachermayer}, Norm attaining operators and renormings of Banach spaces, \emph{Israel J.
Math.} \textbf{44} (1983), 201--212.

\bibitem{Schacher-Pacific}
\textsc{W.~Schachermayer}, Norm attaining operators on some classical Banach spaces, \emph{Pacific J. Math.} \textbf{105} (1983), 427--438.


\end{thebibliography}
\end{document}